\documentclass[12pt,reqno]{amsart}
\usepackage{amsmath}
\usepackage{amssymb, color}
\usepackage[left=3cm,top=3cm,right=3cm,bottom=3cm]{geometry}
\usepackage[usenames,dvipsnames]{xcolor}
\usepackage{graphicx,tikz}
\usetikzlibrary{decorations.pathreplacing}
\usepackage{epsfig}

\begin{document}
\newtheorem{theorem}{Theorem}[section]
\newtheorem{lemma}[theorem]{Lemma}
\newtheorem{claim}[theorem]{Claim}
\newtheorem{definition}[theorem]{Definition}
\newtheorem{conjecture}[theorem]{Conjecture}
\newtheorem{proposition}[theorem]{Proposition}
\newtheorem{algorithm}[theorem]{Algorithm}
\newtheorem{corollary}[theorem]{Corollary}
\newtheorem{observation}[theorem]{Observation}
\newtheorem{problem}[theorem]{Open Problem}
\newcommand{\noin}{\noindent}
\newcommand{\ind}{\indent}
\newcommand{\al}{\alpha}
\newcommand{\om}{\omega}
\newcommand{\pp}{\mathcal P}
\newcommand{\ppp}{\mathfrak P}
\newcommand{\R}{{\mathbb R}}
\newcommand{\N}{{\mathbb N}}
\newcommand\eps{\varepsilon}
\newcommand{\E}{\mathbb E}
\newcommand{\Bin}{\textrm{Bin}}
\newcommand{\Prob}{\mathbb{P}}
\newcommand{\pl}{\textrm{C}}
\newcommand{\dang}{\textrm{dang}}
\renewcommand{\labelenumi}{(\roman{enumi})}
\newcommand{\bc}{\bar c}
\newcommand{\G}{{\mathcal{G}}}
\newcommand{\Po}{{\mathcal{P}}}

\newcommand{\expect}[1]{\E \left [ #1 \right ]}
\newcommand{\floor}[1]{\left \lfloor #1 \right \rfloor}
\newcommand{\ceil}[1]{\left \lceil #1 \right \rceil}
\newcommand{\of}[1]{\left( #1 \right)}
\newcommand{\set}[1]{\left\{ #1 \right\}}
\newcommand{\angs}[1]{\left\langle #1 \right\rangle}
\newcommand{\sqbs}[1]{\left[ #1 \right]}
\newcommand{\sm}{\setminus}
\newcommand{\bfrac}[2]{\of{\frac{#1}{#2}}}
\renewcommand{\k}{\kappa}
\renewcommand{\b}{\beta}
\newcommand{\blue}[1]{{\color{blue} #1}}

\title{The Unit Acquisition Number of Binomial Random Graphs}

\author{Konstantinos Georgiou}
\address{Department of Mathematics, Ryerson University, Toronto, ON, Canada, M5B 2K3}
\thanks{The first and third author are supported in part by NSERC and Ryerson University. The second author is supported in part by OGS}
\email{\texttt{konstantinos@ryerson.ca}}

\author{Somnath Kundu}
\address{Department of Mathematics, Ryerson University, Toronto, ON, Canada, M5B 2K3}
\email{\texttt{somnath.kundu@ryerson.ca}}

\author{Pawe\l{} Pra\l{}at}
\address{Department of Mathematics, Ryerson University, Toronto, ON, Canada, M5B 2K3}
\email{\texttt{pralat@ryerson.ca}}

\begin{abstract}
Let $G$ be a graph in which each vertex initially has weight 1. In each step, the unit weight from a vertex $u$ to a neighbouring vertex $v$ can be moved, provided that the weight on $v$ is at least as large as the weight on $u$. The unit acquisition number of $G$, denoted by $a_u(G)$, is the minimum cardinality of the set of vertices with positive weight at the end of the process (over all acquisition protocols). In this paper, we investigate the Erd\H{o}s-R\'{e}nyi random graph process $(\G(n,m))_{m =0}^{N}$, where $N = {n \choose 2}$. We show that asymptotically almost surely $a_u(\G(n,m)) = 1$ right at the time step the random graph process creates a connected graph. Since trivially $a_u(\G(n,m)) \ge 2$ if the graphs is disconnected, the result holds in the strongest possible sense.
\end{abstract}

\maketitle

\section{Introduction}

Gossiping and broadcasting are two well studied problems involving information dissemination in a group of individuals connected by a communication network~\cite{HHL}. In the gossip problem, each member has a unique piece of information which she would like to pass to everyone else. In the broadcast problem, there is a single piece  of information (starting at one member) which must be passed to every other member of the network. These problems have received attention from mathematicians as well as computer scientists due to their applications in distributed computing~\cite{BGRV}. Gossiping and broadcasting are respectively known as ``all-to-all'' and ``one-to-all'' communication problems. In this paper, we consider the problem of acquisition, which is a type of ``all-to-one'' problem. 

Suppose each vertex of a graph begins with a weight of 1 (this can be thought of as the piece of information starting at that vertex). A \textbf{total acquisition move} is a transfer of all the weight from a vertex $u$ onto a neighbouring vertex $v$, provided that immediately prior to the move, the weight on $v$ is at least the weight on $u$. Suppose a number of total acquisition moves are made until no such moves remain. Such a maximal sequence of moves is referred to as an \textbf{acquisition protocol}  and the vertices which retain positive weight after an acquisition protocol is called a \textbf{residual set}. Note that any residual set is necessarily an independent set. Given a graph $G$, we are interested in the minimum possible size of a residual set and refer to this number as the \textbf{total acquisition number of $G$}, denoted $a_t(G)$.  

\medskip

Other models allow more relaxed consolidation moves and do not require moving all the weight from a vertex $u$ onto a neighbouring vertex $v$ (but it is still required that immediately prior to the move, the weight on $v$ is at least the weight on $u$). A \textbf{unit acquisition move} transfers one unit of weight and a \textbf{fractional acquisition move} allow a fractional (non-integer) amounts of weight to be transferred. Hence, in addition to the total acquisition number $a_t(G)$, the corresponding parameters are the \textbf{unit acquisition number} $a_t(G)$ and the \textbf{fractional acquisition number} $a_f(G)$. Note that unit and then fractional acquisition provide more flexibility in choosing moves than total acquisition does; thus 
\begin{equation}\label{eq:acquisitions}
a_f(G) \le a_u(G) \le a_t(G). 
\end{equation}

The restriction to acquisition moves can be motivated by the so-called ``smaller to larger'' rule in disjoint set data structures. For example, in the UNION-FIND data structure with linked lists, when taking a union, the smaller list should always be appended to the longer list. This heuristic improves the amortized performance over sequences of union operations.  

\medskip

\noindent \textbf{Example:} In order to warm-up with this graph parameter, note that an acquisition protocol for a cycle $C_{4k}$ (for some $k \in \N$) that leaves a residual set of every fourth vertex is the best one can do in any of the three variants of the game; see Figure~\ref{fig:path}. Therefore, $a_f(C_{4k})=a_u(C_{4k})=a_t(C_{4k})=k$.

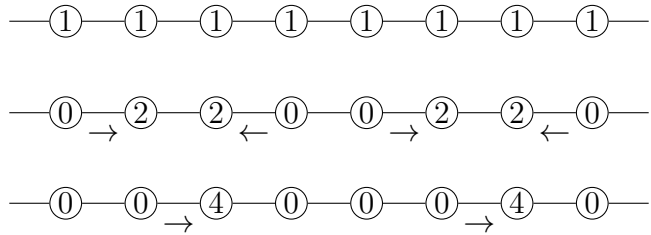
\begin{figure}[ht!]\centering
\begin{tikzpicture}
\draw (-0.75,0) -- (7.75,0);
\foreach \x in {0, 1, 2, 3, 4, 5, 6, 7}{
\filldraw[white] (\x,0) circle (6pt);
\draw (\x,0) circle (6pt);
\draw (\x,0) node {1};
}
\end{tikzpicture}\vskip 0.25 in

\begin{tikzpicture}
\draw (-0.75,0) -- (7.75,0);
\foreach \x in {0, 1, 2, 3, 4, 5, 6, 7}{
\filldraw[white] (\x,0) circle (6pt);
\draw (\x,0) circle (6pt);}
\foreach \x in {1, 2, 5, 6}{
\draw (\x,0) node {2};
}
\foreach \x in {0, 3, 4, 7}{
\draw (\x,0) node {0};
}
\draw (0.5,-0.3) node {$\rightarrow$};
\draw (4.5,-0.3) node {$\rightarrow$};
\draw (2.5,-0.3) node {$\leftarrow$};
\draw (6.5,-0.3) node {$\leftarrow$};
\end{tikzpicture}
\vskip 0.15 in

\begin{tikzpicture}
\draw (-0.75,0) -- (7.75,0);
\foreach \x in {0, 1, 2, 3, 4, 5, 6, 7}{
\filldraw[white] (\x,0) circle (6pt);
\draw (\x,0) circle (6pt);}
\foreach \x in {2,6}{
\draw (\x,0) node {4};
}
\foreach \x in {0, 1,  3, 4, 5, 7}{
\draw (\x,0) node {0};
}
\draw (1.5,-0.3) node {$\rightarrow$};
\draw (5.5,-0.3) node {$\rightarrow$};
\end{tikzpicture}
\caption{The acquisition moves for a fragment of a cycle $C_{4k}$ that leave a residual set of size $k$.} \label{fig:path}
\end{figure}

The parameter $a_t(G)$ was introduced by Lampert and Slater~\cite{LS} and subsequently studied in~\cite{SW, LPWWW, LW, MWW}.  For work on game variations of the parameter and variations where acquisition moves need not transfer the full weight of vertex, see~\cite{Wen, PWW, SW2}. Since in this paper we focus on random structures, we do not comment on these interesting but deterministic results. 

\medskip

Randomness often plays a part in the study of information dissemination problems, usually in the form of a random network or a randomized protocol---see, for example,~\cite{gos, FM, G95}. Before we summarize what is known for the total acquisition number from the perspective of random structures, let us introduce the two models we consider in this paper. The \textbf{binomial random graph} $\G(n,p)$ is a distribution over the class of graphs with vertex set $[n]=\{1, 2, \ldots, n\}$ in which every pair $\{i,j\} \in \binom{[n]}{2}$ appears independently as an edge in $G$ with probability~$p$. Note that $p=p(n)$ may (and usually does) tend to zero as $n$ tends to infinity. We will also consider the \textbf{Erd\H{o}s-R\'{e}nyi random graph process}, which is a stochastic process that starts with $n$ vertices and no edges, and at each step adds one new edge chosen uniformly at random from the set of missing edges. Formally, let $N=\binom{n}{2}$ and let $e_1, e_2,\ldots , e_N$ be a random permutation of the edges of the complete graph $K_n$. The graph process consists of the sequence of random graphs $(\G(n, m))^N_{m=0}$, where $\G(n, m) = ([n], E_m)$ and $E_m = \{e_1, e_2, \ldots, e_m\}$. It is clear that $\G(n, m)$ is a graph taken uniformly at random from the set of all graphs on $n$ vertices and $m$ edges. Finally, we say that an event in a probability space holds \textbf{asymptotically almost surely} (\textbf{a.a.s.}), if its probability tends to one as $n$ goes to infinity. (See, for example,~\cite{JLR} and~\cite{AS} for more details about random graphs.)

\medskip

The total acquisition number of $\G(n,p)$ was studied in~\cite{BBDP}. In particular, LeSaulnier {\em et al.}~\cite{LPWWW} asked for the minimum value of $p=p(n)$ such that a.a.s.\ $a_t(\G(n,p)) = 1$. In~\cite{BBDP} it was proved that $p = \log_2 n/n \approx 1.4427 \ \ln n/n$ is a sharp threshold for this property. Moreover, in the same paper it was also proved that almost all trees $T$ satisfy $a_t(T) = \Theta(n)$, confirming a conjecture of West. 
In~\cite{IMP}, \textbf{random geometric graphs} $\G(n,r)$ were studied in which $n$ vertices are distributed uniformly at random in $[0,\sqrt{n}]^2$ and two vertices being adjacent if and only if their distance is at most $r$. It was proved that asymptotically almost surely $a_t(\G(n,r)) = \Theta( n / (r \log_2 r)^2)$ for the whole range of $r=r(n) \ge 1$ such that $r \log_2 r \le \sqrt{n}$. 
Another way randomness can come into the picture is when initial weights are generated at random. This direction, in particular the case where vertex weights are initially assigned according to independent Poisson distributions of intensity $1$, was considered in~\cite{GKKPZ}.

\medskip

In this paper, we investigate the unit acquisition of $\G(n,p)$. It follows from~(\ref{eq:acquisitions}) and the results from~\cite{BBDP} that a.a.s.\ $a_u(\G(n,p)) = a_t(\G(n,p)) = 1$, provided that $p \ge (1+\eps) \log_2 n / n$ for some $\eps > 0$. However, perhaps surprisingly, it turns out that $p = \ln n / n$ is a sharp threshold for $a_u(\G(n,p)) = 1$---see Corollary~\ref{cor:main}. As a result, this threshold coincides with the threshold for connectivity, which is a trivial lower bound for our property; indeed, if $G$ is disconnected, then $a_u(G) \ge a_f(G) > 1$. In fact, we prove the strongest possible result, that is, we show that $a_u(\G(n,p)) = 1$ holds right at the time step the random graph process creates a connected graph---see Theorem~\ref{thm:main}. On the other hand, it follows from~\cite{BBDP} that at this very moment, a.a.s.\ $a_t(\G(n,p)) > n^{0.3}$ and so there is a drastic difference between the unit acquisition number and the total acquisition counterpart. 

\medskip

Here is our main result. 

\begin{theorem}\label{thm:main}
The following property holds a.a.s. Let $M$ be a random variable defined as follows: 
$$
M = \min \{ m : \G(n, m) \text{ is connected} \}.
$$
Then, 
$$
a_f(\G(n,M)) = a_u(\G(n,M)) = 1.
$$
\end{theorem}

\medskip

Let $\omega = \omega(n)$ be any function tending to infinity as $n \to \infty$. It is well known that a.a.s.\ $m_- \le M \le m_+$, where
$$
m_{-} = \frac {n}{2} \Big( \ln n - \omega \Big) \qquad \text{ and } \qquad m_{+} = \frac {n}{2} \Big( \ln n + \omega \Big).
$$
The two models, $\G(n, p)$ and $\G(n, m)$, are in many cases asymptotically equivalent, provided $\binom{n}{2}p$ is close to $m$. 
In particular, we get immediately the following corollary. 

\begin{corollary}\label{cor:main}
Let $\omega = \omega(n)$ be any function tending to infinity as $n \to \infty$. Let 
$$
p_{-} = \frac {\ln n - \omega}{n} \qquad \text{ and } \qquad p_{+} = \frac {\ln n + \omega}{n}.
$$
\begin{itemize}
\item If $p \le p_-$, then, a.a.s. $a_u(\G(n,p)) \ge a_f(\G(n,p)) \ge 2$.
\item If $p \ge p_+$, then, a.a.s. $a_u(\G(n,p)) = a_f(\G(n,p)) = 1$.
\end{itemize}
\end{corollary}

The paper is structured as follows. In Section~\ref{sec:prep}, we introduce the notation and then prove some useful properties of the $\G(n,p)$ model. The following section, Section~\ref{sec:proof}, is devoted to the proof of the main result.

\section{Notation and Preliminaries}\label{sec:prep}

In this section we give a few preliminary results that will be useful for the proof of our main result, Theorem~\ref{thm:main}. First, we introduce standard asymptotic notation, then we state a specific instance of Chernoff's bound that we will find useful. Finally, we prove some simple and well-known properties that $\G(n,p)$ has that will be used in the proof of the main result.

\subsection{Notation and Convention}\label{sub:notation}
 
Given two functions $f=f(n)$ and $g=g(n)$, we will write $f=O(g)$ if there exists an absolute constant $c$ such that $f
\leq cg$ for all $n$, $f=\Omega(g)$ if $g=O(f)$, $f=\Theta(g)$ if $f=O(g)$ and $f=\Omega(g)$, and we write $f=o(g)$ or $f\ll g$ if the limit $\lim_{n\to\infty} f/g=0$. In addition, we write $f=\omega(g)$ or $f\gg g$ if $g=o(f)$, and unless otherwise specified, $\omega$ will denote an arbitrary function that is $\omega(1)$, assumed to grow slowly. We also will write $f\sim g$ if $f=(1+o(1))g$. 

Through the paper, as typical in the field of random graphs, for expressions that clearly have to be an integer, we round up or down but do not specify which: the choice of which does not affect the argument. 

\subsection{Chernoff's Bound}

We will use the following consequence of Chernoff's bound (see, for example,~\cite{JLR} or~\cite{AS}).

\begin{lemma}
If $X$ is a Binomial random variable $\text{Bin}(k,q)$ with expectation $\mu=kq$, and $0<\eps<1$, then 
$$\Pr[X < (1-\eps)\mu] \le \exp \left( -\frac{\eps^2 \mu}{2} \right),$$ 
and if $\eps > 0$, then
\[\Pr\sqbs{X > (1+\eps)\mu} \le \exp\of{-\frac{\eps^2 \mu}{2+\eps}}.\]
\end{lemma}

\subsection{Typical Properties}

Let $\omega = \omega(n)$ be any function tending to infinity as $n \to \infty$ arbitrarily slowly. In particular, for convenience we will assume that $\omega = o(\ln \ln n)$. Recall that $p_- = p_-(n) = (\ln n - \omega)/n$.  

\medskip

Let us first show that the maximum degree $\Delta$ of $\G(n,p_-)$ is not too far from the average degree that is asymptotic to $np_- \sim \ln n$. This is a well-known result but we prove it for completeness. 

\begin{lemma}\label{lem:max_degree}
A.a.s.\ $\Delta(\G(n,p_-)) \le 4 \ln n$.
\end{lemma}
\begin{proof}
Let $v$ be any vertex in $\G(n,p_-)$. Since $\deg(v)$ is the binomial random variable $\text{Bin}(n-1,p_-)$ with expectation $\mu = (n-1)p_- \sim \ln n$, we get from Chernoff's bound applied with $\eps = 4 \ln n / \mu - 1 \sim 3$ that 
\begin{eqnarray*}
\Pr \big( \deg(v) > 4 \ln n \big) &=& \Pr \Big( \text{Bin} (n-1, p) > (1+\eps) \mu \Big) \\
&\le& \exp \left( - \left( \frac {9}{5} + o(1) \right) \ln n \right) = o(n^{-1}).
\end{eqnarray*}
Hence the expected number of vertices of degree larger than $4 \ln n$ is $o(1)$ and the lemma holds by Markov's inequality. 
\end{proof}

We will also need an upper bound for the number of vertices of a given degree $k \in \N \cup \{0\}$. Note that $k$ is fixed, not a function of $n$.

\begin{lemma}\label{lem:k-degree}
Let $k \in \N \cup \{0\}$. A.a.s.\ the number of vertices of degree $k$ in $\G(n,p_-)$ is at most $(\ln n)^{k+o(1)} \le (\ln n)^{k+1}$. 
\end{lemma}
\begin{proof}
Let us concentrate on any vertex $v \in [n]$. Since by Taylor expansion $1-p = \exp(-p + O(p^2))$, we get that
\begin{eqnarray*}
\Pr \Big( \deg(v) = k \Big) &=& \binom{n-1}{k} \, p_-^k \, \Big(1-p_-\Big)^{n-1-k}\\
&\sim& \frac {n^k}{k!} \left( \frac {\ln n}{n} \right)^k \exp \left( - p_-n + O(p_-^2n) \right)  \\
&\sim& \frac {(\ln n)^k}{k!} \exp \left( - (\ln n - \omega) \right)  \\
&=& \frac {(\ln n)^k e^\omega}{k! \, n}.
\end{eqnarray*}
Hence, the expected number of vertices of degree $k$ is asymptotic to $(\ln n)^k e^{\omega} / k!$. It follows from Markov's inequality that a.a.s.\ the number of vertices of degree $k$ is at most $(\ln n)^k e^{\omega} \omega / k! = (\ln n)^{k+o(1)} \le (\ln n)^{k+1}$, since it is assumed that $\omega = o(\ln \ln n)$. The desired property holds.
\end{proof}

\section{The proof of the main result}\label{sec:proof}

Let us fix $\delta = 0.06$. This parameter is carefully tuned for the argument to hold. It cannot be too large nor too small. We will highlight the two places in the proof where the specific numerical value matters. 
Let $\omega = \omega(n)$ be any function tending to infinity (arbitrarily slowly) as $n \to \infty$. In particular, as before, for convenience we will assume that $\omega = o(\ln \ln n)$. 

Recall that $p_- = p_-(n) = (\ln n - \omega)/n$. 
It is well-known that $\G(n,p_-)$ is disconnected a.a.s. In fact, it is known that a.a.s.\ $\G(n,p_-)$ has the giant component that consists of almost all vertices; the remaining components are isolated vertices. (This fact will also follow from our proof.) As a result, the two models, $\G(n,p)$ and $\G(n,m)$, can be coupled such that $\G(n,p_-)$ is a subgraph of $\G(n,M)$. It is easier to work with $\G(n,p_-)$ rather than with $\G(n,M)$ since in $\G(n,p_-)$ edges occur independently. As a result, we will mostly use the former moving to the latter only for a brief moment to deal with isolated vertices that are present in $\G(n,p_-)$. 

\subsection{Big Picture}

Our strategy is to build a rooted spanning tree of the giant component of $\G(n,p_-)$. We do it in a few phases. We first build a tree that spans roughly $(1-\delta)n$ vertices that form set $\mathcal{T}$ (Subsection~\ref{sec:tree}). The remaining vertices form set $\mathcal{R}$ that needs to be carefully partitioned to prepare it to be attached to the tree (Subsection~\ref{sec:partitioning}). At this point we move our attention to isolated vertices (Subsection~\ref{sec:isolated_vertices}). We continue the Erd\H{o}s-R\'enyi random graph process ignoring all incoming edges unless they are incident to one of the isolated vertices. At time $M$ when $\G(n,m)$ becomes connected, all isolated vertices from $\G(n,p_-)$ are adjacent to at least one neighbour in the giant component of $\G(n,p_-)$. At this point, we come back to building the spanning tree of the giant component of $\G(n,p_-)$ and carefully attach the remaining vertices to the tree (Subsection~\ref{sec:remaining_vertices}). The rooted spanning tree of $\G(n,M)$ is then formed and it remains to show that there exists an acquisition protocol that uses only edges of the tree and results in a residual set consisting of the root (Subsection~\ref{sec:acquisition}).

\medskip

The argument is fairly long and in a few places quite delicate. In order to help the reader follow it, we highlight the most important conclusions as independent claims. 

\subsection{Building a tree on at least $(1-\delta)n$ vertices}\label{sec:tree}

Let us start with any vertex $v_0$ in $\G(n,p_-)$ that will become the \textbf{root} of the final spanning tree. We will apply the following ``\textbf{breadth first search}'' (\textbf{BFS}) type algorithm to build a tree. The vertex set $[n]$ will always be partitioned into two sets: vertices that are \textbf{discovered} and \textbf{non-discovered}. We initiate the process by putting $v_0$ into the \textbf{queue} $Q$ and assigning the status discovered to $v_0$; the remaining vertices are non-discovered. In each step of the process, we remove vertex $v$ from the queue $Q$ and expose edges from $v$ to all non-discovered vertices. These new neighbours of $v$ (non-discovered vertices at this point) are called \textbf{children} of $v$ and $v$ itself is a \textbf{parent} for them. We will additionally label $v$ as \textbf{good} if it has at least $\frac {\delta}{2} \ln n$ children; otherwise, $v$ is labelled as \textbf{bad}. If vertex $v$ is good, then we arbitrarily select $\frac {\delta}{4} \ln n$ of its children and label them as \textbf{good whiskers}. The remaining children (there are at least $\frac {\delta}{4} \ln n$ of them) are put into the queue $Q$. After that all children of $v$ change their status to discovered. On the other hand, if vertex $v$ is bad, then we label all of its children as \textbf{bad whiskers} and change their status to discovered. Note that whiskers (regardless whether good or bad) are not put into the queue and so they will become leaves in the tree. We continue the process until the number of non-discovered vertices drops below $\delta n$ or $Q$ becomes prematurely empty. 

\medskip

Consider any vertex $v$ that was removed from the queue $Q$ at some point of the process. Since during the entire process the number of non-discovered vertices is always at least $\delta n$, the number of children of $v$ is stochastically bounded from below by the binomial random variable $\text{Bin}(\delta n,p_-)$ with expectation $\mu \sim \delta n p_- \sim \delta \ln n$. It follows from Chernoff's bound applied with $\eps = 1 - \delta \ln n / (2 \mu) \sim 1/2$ that
\begin{eqnarray} 
\Pr ( v \text{ is bad}) &\le& \Pr \left( \text{Bin}(\delta n, p_-) < \frac {\delta}{2} \ln n \right) =  \Pr \Big( \text{Bin}(\delta n, p_-) < (1-\eps)\mu \Big)  \nonumber \\ 
&\le& \exp \left( - \left( \frac {1}{8} + o(1) \right) \mu \right)= \exp \left( - \left( \frac {\delta}{8} + o(1) \right) \ln n \right) \nonumber \\ 
&\le& n^{-\delta/9}. \label{eq:pr_bad}
\end{eqnarray}
In particular, since $n^{-\delta/9} = o(1)$, we get the following observation.

\begin{claim}\label{claim:root_is_good}
A.a.s.\ the root $v_0$ is good.
\end{claim}

Consider now any good vertex $v$. By Lemma~\ref{lem:max_degree}, we may assume that $v$ has at most $4 \ln n$ children. Using~(\ref{eq:pr_bad}) we get that the probability that $v$ has at least $\lceil 10 / \delta \rceil$ bad children is at most 
\begin{eqnarray*}
\binom{4 \ln n}{\lceil 10 / \delta \rceil} \Pr ( v \text{ is bad})^{\lceil 10 / \delta \rceil} &\le& \left( \frac {4 e \ln n}{ \lceil 10 / \delta \rceil n^{\delta/9} } \right)^{ \lceil 10 / \delta \rceil }\\ 
&=& n^{ -\lceil 10 / \delta \rceil \, \delta / 9 +o(1) } \le n^{-10/9 + o(1)} = o(n^{-1}).
\end{eqnarray*}
Trivially, the number of good vertices is at most $n$. We get that the expected number of good vertices with many bad children tends to zero and so the following claim holds.

\begin{claim}\label{claim:good-vertices}
A.a.s.\ no good vertex has at least $\lceil 10 / \delta \rceil$ bad children. 
\end{claim}

Recall that, by definition, each good vertex has at least $\frac {\delta}{2} \ln n$ children but only $\frac {\delta}{4} \ln n$ of them are good whiskers. Hence, by Claim~\ref{claim:good-vertices}, we may assume that all good vertices have at least $\frac {\delta}{4} \ln n - O(1) \ge \frac {\delta}{5} \ln n$ good children (unless the BFS process stops naturally once $(1-\delta)$ fraction of all vertices are discovered). Combining it with Claim~\ref{claim:root_is_good} we get the next observation. 

\begin{claim}
A.a.s.\ the BFS process does not finish prematurely, that is, the queue $Q$ never becomes empty and the process ends naturally once the number of non-discovered vertices drops below $\delta n$.
\end{claim}

It follows from this claim that when the BFS process stops the queue $Q$ is not empty. We will label the vertices in the queue as good whiskers. For convenience, let us compile a list of properties of the tree we just created. 

\begin{definition}
The vertex set $[n]$ is partitioned into two sets: $\mathcal{T}$ (vertices of the BFS tree) is the set of discovered vertices and $\mathcal{R}$ (remaining vertices) is the set of non-discovered vertices. 
\end{definition}

\begin{claim}\label{claim:summary_T}
The following properties hold a.a.s.:
\begin{enumerate}
\item $|\mathcal{T}| = (1-\delta)n + O(\ln n) \sim (1-\delta) n$.
\item $|\mathcal{R}| = \delta n + O(\ln n) \sim \delta n$.
\item There are $O(n^{1-\delta/9}) = o(n)$ bad vertices.
\item There are $O(n^{1-\delta/9} \ln n) = o(n)$ bad whiskers.
\item There are $\Theta(n / \ln n) = o(n)$ good vertices. Each good vertex has at least $\frac {\delta}{4} \ln n$ good whiskers.
\item There are $(1-\delta+o(1))n \sim |\mathcal{T}|$ good whiskers, that is, almost all vertices of $\mathcal{T}$ are good whiskers.
\item All edges between good/bad vertices and $\mathcal{R}$ are exposed (and no edge was found). 
\item No edge between good/bad whiskers and $\mathcal{R}$ is exposed.
\item The height of the rooted tree on vertices from $\mathcal{T}$ has height $(1+o(1)) \ln n / \ln \ln n$.
\end{enumerate}
\end{claim}

Since a.a.s.\ the maximum degree in $\G(n,p_-)$ is at most $4 \ln n$ (Lemma~\ref{lem:max_degree}), when the process stops the number of non-discovered vertices drops below $\delta n$ but it is very close to that value. Part~(i) and~(ii) follow. In order to see part~(iii), note that by~(\ref{eq:pr_bad}), the expected number of bad vertices is at most $n^{1-\delta/9}$ and it follows from Chernoff's bound that a.a.s.\ it is at most, say, $2n^{1-\delta/9}$. Part~(iv) follows again from the observation on the maximum degree. Since the number of bad vertices/whiskers is negligible, almost all vertices of $\mathcal{T}$ are either good vertices or good whiskers. By definition, each good vertex has at least $\frac {\delta}{4} \ln n$ good whiskers. Combining these two observations, we get that in fact almost all vertices of $\mathcal{T}$ are good whiskers, part~(vi), and so there are $\Theta(n / \ln n)$ good vertices, part (v). Parts~(vii) and~(viii) follow immediately from the definition of the process. Finally, let us estimate the height of the tree on vertices from $\mathcal{T}$. As pointed out earlier, we may assume that each good vertex has at least $\frac {\delta}{5} \ln n$ good children (except the last few levels of the tree). Hence, the height of the tree is at most
$$
\frac {\ln n}{\ln (\frac {\delta}{5} \ln n)} + O(1) = \frac {\ln n}{\ln \ln n-\ln(5/\delta)} + O(1) \sim \frac {\ln n}{\ln \ln n}.
$$
On the other hand, since the maximum degree in $\G(n,p_-)$ is a.a.s.\ $4 \ln n$, we get an asymptotic matching lower bound and so part~(ix) follows.
 
\subsection{Partitioning vertices of $\mathcal{R}$}\label{sec:partitioning}

It will be very important in which order we expose the remaining edges of $\G(n,p_-)$. Moreover, we will often expose only partial information; for example, we may want to reveal the degree of a given vertex without exposing where its neighbours actually are. 

\medskip

For each vertex $v \in \mathcal{R}$ we expose $\deg_{\mathcal{T}}(v)$, the number of neighbours in $\mathcal{T}$. If $\deg_{\mathcal{T}}(v) \ge \frac {1-\delta}{1.01} \ln n$, then $v$ is called \textbf{high degree} vertex. If $\deg_{\mathcal{T}}(v) \le \lceil 10/\delta \rceil$, then $v$ is called \textbf{low degree} vertex. The remaining vertices of $\mathcal{R}$ are called \textbf{medium degree} vertices. 
By Claim~\ref{claim:summary_T}~(vi--viii), $\deg_{\mathcal{T}}(v) \in \Bin(w,p_-)$, where $w \sim (1-\delta) n$ is the number of whiskers in $\mathcal{T}$. By Chernoff's bound applied with $\mu = wp_- \sim (1-\delta) \ln n$ and 
$$
\eps = 1 - \frac {(1-\delta) \ln n}{1.01 \mu} \sim 1 - \frac {1}{1.01} = \frac {0.01}{1.01}
$$ 
we get that
\begin{eqnarray*}
\Pr \Big( \deg_{\mathcal{T}}(v) < \frac {1-\delta}{1.01} \ln n \Big) &=& \Pr \Big( \Bin(w,p_-) < (1-\eps) \mu \Big) \le n^{-(1-\delta)/10^5}.
\end{eqnarray*}
Hence, the expected number of vertices of medium or low degree is at most $n^{1-(1-\delta)/10^5}$. Let us now fix $k \in \N \cup \{0\}$. 
\begin{eqnarray*}
\Pr \Big( \deg_{\mathcal{T}}(v) =k \Big) &=& \binom{w}{k} p_-^k (1-p_-)^{w-k} \\
&\sim& \frac { \big( (1-\delta)n \big)^k}{k!} \left( \frac {\ln n}{n} \right)^k \exp \left( - (1+o(1)) \ \frac {\ln n}{n} \ (1-\delta) n \right)\\
&=& n^{-1 + \delta +o(1)}. 
\end{eqnarray*}
Hence, the expected number of vertices of low degree is at most $\sum_{k=0}^{\lceil 10/\delta \rceil} n^{\delta+o(1)} = n^{\delta+o(1)}$. Combining the two observations together, we get the following claim by Markov's inequality.

\begin{claim}\label{claim:medium_and_low}
The following properties hold a.a.s.:
\begin{enumerate}
\item At most $n^{1-(1-\delta)/10^6}$ vertices in $\mathcal{R}$ are of medium or low degree.
\item At most $n^{\delta+o(1)}$ vertices in $\mathcal{R}$ are of low degree.
\end{enumerate}
\end{claim}

Vertices of high and medium degree do not cause any problems, they can be appropriately attached to the tree. We will do it in Subsection~\ref{sec:remaining_vertices}. However, we need to pay attention to low degree vertices. For each low degree vertex $v \in \mathcal{R}$ we expose $\deg_{\mathcal{R}}(v)$, the number of neighbours in $\mathcal{R}$.

\medskip

For each low degree vertex $v \in \mathcal{R}$ with $\deg_{\mathcal{R}}(v) \ge 10^7/(1-\delta)$, we expose all neighbours in $\mathcal{R}$. By Claim~\ref{claim:medium_and_low}~(i), the probability that none of them is of high degree is at most
$$
\left( \frac {n^{1-(1-\delta)/10^6}}{ |\mathcal{R}| } \right)^{10^7/(1-\delta)} = \left( \frac {n^{-(1-\delta)/10^6}}{ (\delta+o(1)) } \right)^{10^7/(1-\delta)} = O(n^{-10}) = o(n^{-1}).
$$
Hence, by Markov's inequality, a.a.s.\ each low degree vertex $v$ of this type has a high degree neighbour that we call its \textbf{parent} and denote it by $P(v)$. Vertex $v$ itself is called \textbf{dangerous}. 

The remaining low degree vertices have less than $10^7/(1-\delta)$ neighbours in $\mathcal{R}$ and at most $\lceil 10/\delta \rceil$ neighbours in $\mathcal{T}$ (by the definition of being of low degree). Hence, their degrees are at most $C-2$, where $C = \lceil 10/\delta \rceil + 1 + 10^7/(1-\delta)$. By Lemma~\ref{lem:k-degree}, we may assume that there are at most $\sum_{k=0}^{C-2} (\ln n)^{k+1} \le (\ln n)^{C}$ of them. 

If $v$ has at least one neighbour in $\mathcal{R}$, then we simply expose the location of that neighbour, call it a \textbf{parent} of $v$ ($P(v)$), and $v$ itself becomes \textbf{dangerous}. Since the expected number of parents that are not of high degree is at most
$$
(\ln n)^{C} \cdot \frac {n^{1-(1-\delta)/10^6}}{ |\mathcal{R}| } = n^{-(1-\delta)/10^6+o(1)} = o(1),
$$
a.a.s.\ all parents are of high degree. 

Suppose now that $v$ has no neighbour in $\mathcal{R}$ but has at least one neighbour in $\mathcal{T}$. As before, we simply expose the location of that neighbour and attach $v$ to the tree through that neighbour that, as usual, is called a \textbf{parent} of $v$ and denoted $P(v)$. There are two things that need to be checked. First, observe that by Claim~\ref{claim:summary_T}~(iv) and (vi), the expected number of vertices that get attached to some bad whisker is at most
$$
(\ln n)^{C} \cdot \frac { O(n^{1-\delta/9} \ln n) } { | \mathcal{T} | } \le n^{-\delta/9+o(1)} = o(1).
$$
Hence, a.a.s.\ all vertices are attached to good whiskers. Moreover, the expected number of good vertices in $\mathcal{T}$ with at least 2 of its good whiskers attached to some low degree vertex is at most
$$
n \cdot \binom{(\ln n)^C}{2} \cdot \left( \frac {4 \ln n} { (1-\delta+o(1)) n } \right)^2 = n^{-1+o(1)} = o(1).
$$
Indeed, there are at most $n$ choices for good vertices, at most $\binom{(\ln n)^C}{2}$ choices for low degree vertices, and $\frac {4 \ln n} { (1-\delta+o(1)) n }$ is an upper bound for the probability that a given low degree vertex is adjacent to a whisker of the selected good vertex. 

\medskip

Combining all observations together we arrive with the following claim.

\begin{claim}\label{claim:parents_of_low_degree}
A.a.s.\ the following properties hold. Each vertex $v$ in $\mathcal{R}$ that is of low degree satisfies one of the following properties:
\begin{enumerate}
\item $v$ has a parent $P(v)$ in $\mathcal{R}$ that is of high degree.
\item $v$ has a parent $P(v)$ in $\mathcal{T}$ that is a good whisker. 
\item $v$ is an isolated vertex in $\G(n,p_-)$. 
\end{enumerate} 
Moreover, good vertices have at most one of their good whiskers attached to some low degree vertex.
\end{claim}

\subsection{Dealing with Isolated Vertices}\label{sec:isolated_vertices}

Low degree vertices in $\mathcal{R}$ that are not isolated are already attached to their parents. Dangerous vertices have their parents in $\mathcal{R}$ that are of high degree. The remaining low degree vertices have their parents already in $\mathcal{T}$. It is time to deal with isolated vertices that form a set $\mathcal{I} \subseteq \mathcal{R}$ that are present in $\G(n,p)$. By Lemma~\ref{lem:k-degree}, we may assume that $|\mathcal{I} | \le (\ln n)^{o(1)}$.

Recall that the binomial random graph and the Erd\H{o}s-R\'enyi process are coupled such that $\G(n,p_-) \subseteq \G(n,M)$. We may then simply start the process from $\G(n,p_-)$ and continue until we reach $\G(n,M)$, that is, when the last isolated vertex disappears. We ignore all incoming edges unless they are adjacent to one of the vertices in $\mathcal{I}$. If one of the endpoints of an edge is $v \in \mathcal{I}$ that is still isolated, then the other endpoint becomes a \textbf{parent} $P(v)$ of $v$. Since $v$ is isolated, its parent is a random vertex taken uniformly at random from $[n] \setminus \{v\}$. The expected number of parents created this way that are either bad whiskers in $\mathcal{T}$ or medium/low degree in $\mathcal{R}$ is, by Claim~\ref{claim:summary_T}~(iv) and Claim~\ref{claim:medium_and_low}~(i), at most
$$
(\ln n)^{o(1)} \cdot \frac {O(n^{1-\delta/9} \ln n) + O(n^{1-(1-\delta)/10^6})}{n-1} \le n^{-\min\{\delta/9, (1-\delta)/10^6\}+o(1)} = o(1).
$$
We get the following claim by Markov's inequality.

\begin{claim}\label{claim:parents_of_isolated}
A.a.s.\ the following properties hold. There are at most $(\ln n)^{o(1)}$ isolated vertices in $\G(n,p_-)$. Each isolated vertex $v$ has a parent $P(v)$ identified that is either a good whisker in $\mathcal{T}$ or a high degree vertex in $\mathcal{R}$.
\end{claim}

As usual, if an isolated vertex $v$ has a parent in $\mathcal{R}$, then it is called \textbf{dangerous}. 

\subsection{Connecting Remaining Vertices}\label{sec:remaining_vertices}

We are now back to building a spanning rooted tree of the giant component of $\G(n,p_-)$. All low degree vertices in $\mathcal{R}$ are already dealt with. It remains to attach high and medium degree vertices to $\mathcal{T}$. They need to be appended to some good whisker but we will also need to make sure that no good vertex has all of its good whiskers selected by some vertex in $\mathcal{R}$. Hence, we cannot blindly connect all remaining vertices in $\mathcal{R}$ to some good whiskers (as it would create a problem a.a.s.) but rather we need to do it carefully by selecting a perfect matching between high and medium degree vertices in $\mathcal{R}$ and pairs of good whiskers. Moreover, parents of dangerous vertices will have to be attached to some special places so we need to put them aside for a moment and deal with them later. The reason for all of these restrictions will become clear in Subsection~\ref{sec:acquisition}.

Let $R \subseteq \mathcal{R}$ be the set of high and medium degree vertices in $\mathcal{R}$ that are not parents of any dangerous vertices. Our goal is to attach all vertices from $R$ to good whiskers in $\mathcal{T}$. By definition,  each vertex in $R$ has at least $\lceil 10/\delta \rceil+1$ neighbours in $\mathcal{T}$. We expose the information whether these neighbours are bad or good whiskers but we do not expose their exact locations yet. The probability that a given vertex $v \in R$ has at least $\lceil 10/\delta \rceil$ neighbours being bad whiskers is, by Claim~\ref{claim:summary_T}~(i) and~(iv)
\begin{align*}
\binom{\lceil 10/\delta \rceil+1}{\lceil 10/\delta \rceil} & \left( \frac {O(n^{1-\delta/9} \ln n)}{|\mathcal{T}|} \right)^{\lceil 10/\delta \rceil} = O(1) \cdot \left( n^{-\delta/9} \ln n \right)^{\lceil 10/\delta \rceil} \\
&= n^{-(\delta/9)\lceil 10/\delta \rceil + o(1)} \le n^{-10/9 + o(1)} = o(n^{-1}).
\end{align*}
Hence, we get the next claim by Markov's inequality. 

\begin{claim}\label{claim:two_neighbours}
A.a.s.\ each vertex in $R \subseteq \mathcal{R}$ has at least two neighbours in $\mathcal{T}$ that are good whiskers.
\end{claim}

Now, it is time to connect vertices from $R$ to good whiskers that are partitioned into \textbf{buckets}. Every bucket consists of two good whiskers (except possibly one bucket, if the number of good whiskers is odd). To get the desired partition, we investigate good vertices, one by one, and assign their good whiskers into buckets, leaving at most one good whisker unassigned per good vertex. Then, we arbitrarily put the remaining good whiskers into buckets leaving at most one good whisker that will have its own bucket.

We expose the two neighbours of vertices in $R$ that are guaranteed to exist by Claim~\ref{claim:two_neighbours}. Our goal is to show that there exists a matching between set $R$ and buckets that saturates $R$. In order to do that, we prove that the Hall's condition (both necessary and sufficient condition for the desired perfect matching to exist) holds a.a.s.: for all $X \subseteq R$, we have $|N(X)| \ge |X|$, where $N(X)$ is the set of buckets the two edges from $X$ are incident to.

The Hall's condition fails if, for some value of $k$ such that $2 \le k \le |R|$, there exists a set $X \subseteq R$ of $k$ vertices and a set $Y$ of $k-1$ buckets such that all neighbours of $X$ are in $Y$, that is, $N(X) \subseteq Y$. By Claims~\ref{claim:summary_T}~(ii), \ref{claim:medium_and_low}~(i), and~\ref{claim:parents_of_isolated}, we may assume that $|R| \sim |\mathcal{R}| \sim \delta n$. Hence, there are $\binom{(\delta+o(1))n}{k}$ choices for $X$. By Claim~\ref{claim:summary_T}~(vi), we may assume that there are $\binom{((1-\delta)/2+o(1))n}{k-1}$ choices for $Y$. The probability that both neighbours of a given vertex in $X$ are in $Y$ is at most $( 2(k-1) / ((1-\delta+o(1))n) )^2$. Hence, the expected number of sets $X \subseteq R$ of size $k$ for which the condition fails is at most
\begin{align*}
\xi_k &= \binom{(\delta+o(1))n}{k} \binom{( \frac{1-\delta}{2}+o(1))n}{k-1} \left( \frac {2(k-1)}{(1-\delta+o(1))n} \right)^{2k} \\
&\le \left( \frac {e(\delta+o(1))n}{k} \right)^k \left( \frac{ e( \frac{1-\delta}{2}+o(1))n}{k-1} \right)^{k-1} \left( \frac {2(k-1)}{(1-\delta+o(1))n} \right)^{2k} \\
&\le \frac {1}{n} \left( \frac {2 e^2 \delta}{1-\delta} +o(1) \right)^k \ \frac {(k-1)^{k+1}}{k^k}.
\end{align*}
Observe that 
$$
\frac {(k-1)^{k+1}}{k^k} = (k-1) \left( 1 - \frac {1}{k} \right)^k = O(k),
$$
and so
$$
\xi_k = \frac {O(k)}{n} \left( \frac {2 e^2 \delta}{1-\delta} +o(1) \right)^k.
$$
Since $\delta = 0.06$, we get that $(2 e^2 \delta)/(1-\delta) < 0.95$. (This is the first time when the numerical value of $\delta$ matters. For the argument to hold, it has to be small enough.) It follows that the expected number of sets $X \subseteq R$ of any size for which the condition fails is at most
$$
\sum_{k = 2}^{|R|} \xi_k = \sum_{k = 2}^{|R|} \frac {O(k)}{n} \ 0.95^k \le \sum_{k = 2}^{(\ln n)^2} \frac {O(k)}{n} + \sum_{k = (\ln n)^2}^{|R|} O \left( 0.95^{(\ln n)^2} \right) = o(1).
$$
Hence, by Markov's inequality, a.a.s.\ the Hall's condition holds and we get the next claim.

\begin{claim}\label{claim:matching_saturating_R}
A.a.s.\ there exists a matching between set $R$ and the set of buckets that saturates $R$.
\end{claim}

We connect vertices in $R$ to the tree using the edges of the matching. Good whiskers that are associated with vertices in $R$ are called \textbf{lucky}. Since it is a matching saturating $R$, the number of lucky whiskers is equal to $|R| \sim \delta n$. It remains to connect parents of dangerous vertices (that are in $\mathcal{R}$) to the tree. By Claim~\ref{claim:parents_of_low_degree} and~\ref{claim:parents_of_isolated}, these parents are of high degree, that is, they have at least $\frac {1-\delta}{1.01} \ln n$ neighbours in $\mathcal{T}$. Since dangerous vertices are of low degree, by Claim~\ref{claim:medium_and_low}~(ii) there are at most $n^{\delta+o(1)}$ parents to deal with. Our goal is to connect them to lucky whiskers. As usual, the reason for this will be clear in Subsection~\ref{sec:acquisition}.

We expose edges from parents of dangerous vertices to $\mathcal{T}$. The expected number of them with no lucky neighbour is at most
\begin{align*}
n^{\delta+o(1)} & \left( 1 - (1+o(1)) \frac {|R|}{|\mathcal{T}|} \right)^{\frac {1-\delta}{1.01} \ln n} \\
&= \exp \left( \left( \delta + \frac {1-\delta}{1.01} \cdot \ln \left( 1-\frac {\delta}{1-\delta} \right) + o(1) \right) \ln n \right) \\
&\le \exp( - 0.001 \ln n ) = o(1).
\end{align*}
(This is the second time when the numerical value of $\delta$ matters. For the argument to hold, $\delta$ has to be large enough. A careful reader probably noticed that a typical vertex in $\mathcal{R}$ has no lucky neighbour with probability $(1-p_-)^{|R|} = n^{-\delta+o(1)}$ and so the argument would not work if dangerous vertices simply select any vertex from $\mathcal{R}$ as a parent. That was the reason we carefully selected them to be of high degree.) As usual, Markov's inequality is enough to get the last claim. This concludes our tedious process of constructing the rooted tree. 

\begin{claim}
A.a.s.\ parents of dangerous vertices are attached to lucky whiskers in $\mathcal{T}$.
\end{claim}

\subsection{Acquisition Protocol on the Tree}\label{sec:acquisition}

In previous subsections, we created a rooted spanning tree $T$ that is a subgraph of $\G(n,M)$. Our last task is to use edges of $T$ to preform an acquisition protocol that yields a residual set consisting only of the root $v_0$ of $T$. It will imply that $a_u(\G(n,M)) \le a_u(T) = 1$ and so it will finish the proof.

Since tree $T$ is rooted at vertex $v_0$ we may introduce \textbf{levels} depending on the distance from the root. Whiskers that are at the greatest distance from $v_0$ are on level 3 and the closer to the root vertices are, the larger the corresponding level is. Let us summarize the structure of the tree before we prove that its unit acquisition number is equal to 1. See Figure~\ref{fig:tree} for an illustration. 

\begin{enumerate}
\item The root, vertex $v_0$, is at the highest level $h \sim \ln n / \ln \ln n$. (By Claim~\ref{claim:summary_T}~(ix).)
\item All good vertices are on levels between 4 and $h$.
\item All good vertices have at least $(\frac {\delta}{8} + o(1)) \ln n \gg h$ children that are leaves. (Indeed, by definition each good vertex has at least $\frac {\delta}{4} \ln n$ good whiskers. By Claim~\ref{claim:parents_of_low_degree}, at most one of them is attached to low degree vertex. By Claim~\ref{claim:parents_of_isolated}, at most $(\ln n)^{o(1)}$ of them are attached to vertices that were isolated in $\G(n,p_-)$. By Claim~\ref{claim:matching_saturating_R}, at most $\frac {\delta}{8} \ln n + 1$ of them are lucky, that is, attached to vertices in $R$.)
\item All bad vertices are either leaves or their children are leaves. 
\item All vertices in $\mathcal{R}$ are at distance 1 or 2 from some good whisker. Those that are distance 2 (dangerous vertices) are connected through their parents to some good and lucky whiskers. By definition, that lucky whisker has at least one leaf attached.
\end{enumerate}

\begin{figure}[ht]
\centering
\includegraphics[angle=0,width=14cm]{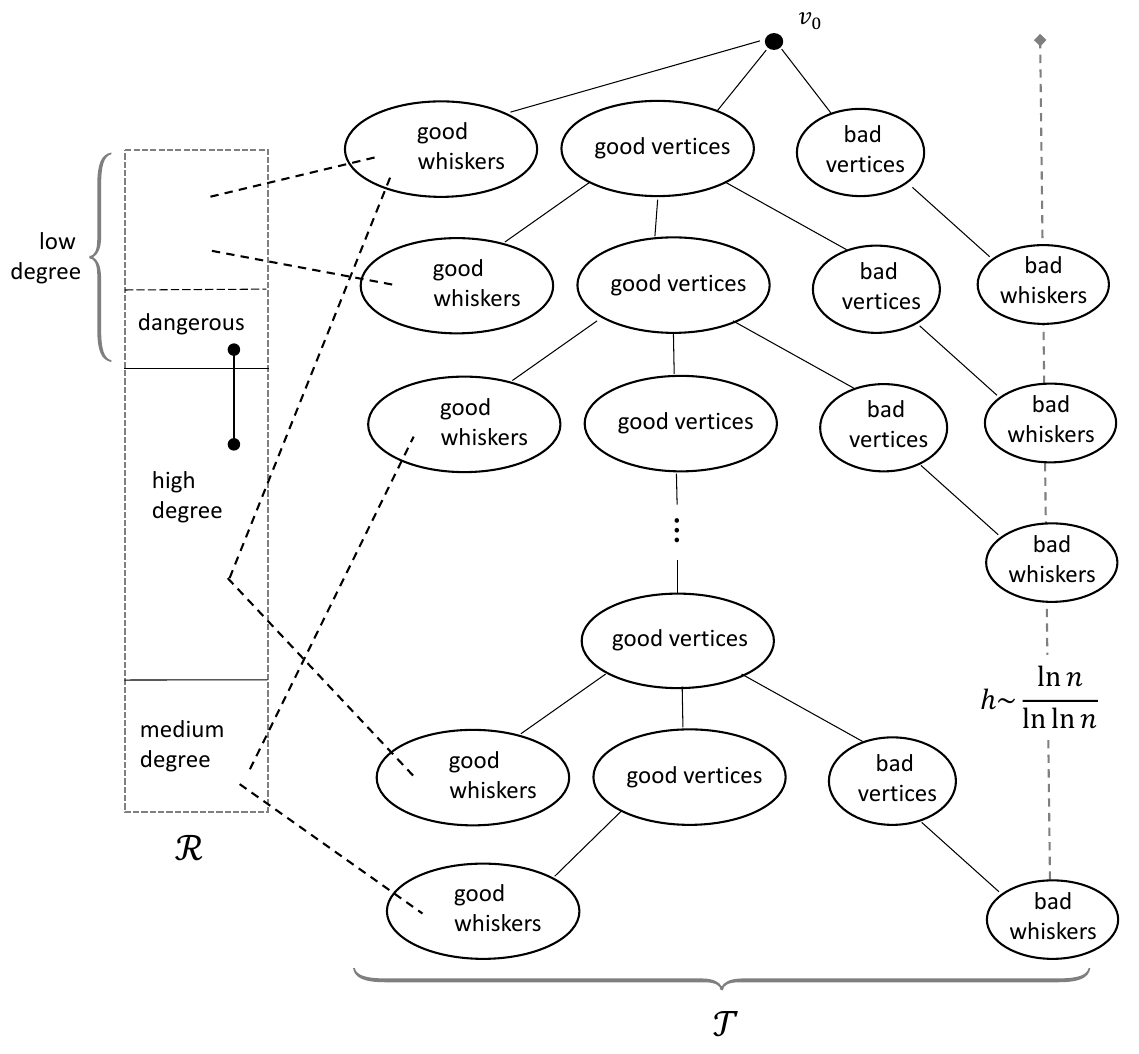}
\caption{Rooted spanning tree of $\G(n,M)$.}
\label{fig:tree}
\end{figure}

The acquisition protocol works as follows. Because of properties~(i-iii), we may move the weight from some children of good vertices that are leaves so that each good vertex on level $k$ has weight equal to $k$. If a bad vertex has at least one leaf attached, then we move one unit from an arbitrary leaf to this bad vertex. By property~(iv), each bad vertex is either a leaf or has weight equal to 2. If a good whisker has at least one leaf attached, then we move one unit from one of the leaves so that this good whisker has weight equal to 2. In particular, all lucky whiskers have weight 2. If a vertex $v \in \mathcal{R}$ is a parent of some dangerous vertex, then we move the weight from that vertex to $v$ and then move one unit from $v$ to the corresponding lucky good whisker. After that operation, the lucky good whisker has weight equal to 3. If $v$ is a parent to another dangerous vertex, we pick one of them arbitrarily and move its weight to $v$.

After that operation, our job is easy as the rooted and weighted tree has a very nice property. For any edge $uv$ in the tree ($u$ is closer to the root than $v$), the weight on $u$ is larger than the weight on $v$. It is straightforward to see that one can now move all weight to the root $v_0$. Indeed, in each step, one can consider the tree induced by vertices with non-zero weight, pick an arbitrary leaf, and move one unit from there all the way up to the root. We may repeat this step until all weight is accumulated on $v_0$. This finishes the proof of the main theorem and the paper.

\end{document}